\documentclass[twoside,12pt]{article}
\usepackage[all, 2cell, dvips]{xy}
\usepackage{latexsym,amsmath, amsfonts, graphics, epsf, epic}
\usepackage{amsthm}
\usepackage{amssymb}
\usepackage{amsopn}
\usepackage{amscd}
\usepackage{color}

\theoremstyle{plain}
\newtheorem{thm}{Theorem}[section]
\newtheorem{cor}[thm]{Corollary}
\newtheorem{lem}[thm]{Lemma}
\newtheorem{prop}[thm]{Proposition}

\theoremstyle{definition}

\newtheorem{remark}[thm]{Remark}
\newtheorem{example}[thm]{Example}
\newtheorem{defn}[thm]{Definition}

\newtheorem{conjecture}[thm]{Conjecture}

\usepackage{amssymb}

\def\sg{\sigma}

\def\lm{\lambda}

\def\al{\alpha}

\def\l.l.o.{\it l.l.o}

\def\f{\varphi}

\def\chiup{\raise 2pt\hbox{$\chi$}}

\title{Kronecker multiplicities in the $(k,\ell)$ hook are polynomially bounded}

\author{Amitai Regev}

\begin{document}

\maketitle

{\bf Abstract}. The problem of decomposing the Kronecker product of $S_n$
characters is one of the last major open problems
in the ordinary representation theory of the symmetric group $S_n$.
Here we prove upper and lower polynomial bounds for the multiplicities of the 
Kronecker product $\chi^\lm\otimes\chi^\mu$,
where for some fixed $k$ and $\ell$
both partitions $\lm$ and $\mu$ are in the $(k,\ell)$ hook, 
$\lm$ and $\mu$ are partitions of $n$, and $n$ goes to infinity.

\medskip
~~~~~~~~~~2010 Mathematics Subject Classification:  Primary 20C30, Secondary 05A17

\section{Introduction}
We assume that the characteristic of the base field is zero: $char(F)=0$. As usual
 $S_n$ is the $n$-th symmetric group. Let $\lm$ be a partition of $n$, $\lm\vdash n$,
 then $\lm$ corresponds to the irreducible $S_n$ character $\chi^\lm$, and all the
 irreducible $S_n$ characters are of that form $\chi^\lm$~\cite{kerber, macdonald, sagan, stanley}.
 Let $\f,\psi$ be two $S_n$ characters (same $n$).
 Their Kronecker --
 or {\it inner} tensor -- product $\f\otimes\psi$ is defined via $(\f\otimes\psi)(\sg)=\f(\sg)\cdot\psi(\sg)$
 where $\sg\in S_n$. Then $\f\otimes\psi$ is an $S_n$ character, and since $char(F)=0$, $\f\otimes\psi$
 is a (non-negative) integer combination of the irreducibles $\chi^\lm$.
In fact, the same construction and decomposition problem exist -- for any finite group.

\begin{defn}
Let $\lm,\mu\vdash n$, let 
$\chi^\lm\otimes\chi^\mu$ denote the Kronecker product of $\chi^\lm$ and $\chi^\mu$ and write
\begin{eqnarray}\label{october.9}
\chi^{\lm}\otimes \chi^{\mu}=\sum_{\rho\vdash n} \kappa(\lm,\mu,\rho)\cdot\chi^\rho.
\end{eqnarray}
This equation defines the multiplicities $\kappa(\lm,\mu,\rho)$.
Thus $\kappa(\lm,\mu,\rho)$ is the multiplicity of $\chi^\rho$ in $\chi^\lm\otimes\chi^\mu$.
We call the coefficients $\kappa(\lm,\mu,\rho)$ {\it the Kronecker multiplicities}.
\end{defn}

Algorithms for calculating the multiplicities $\kappa(\lm,\mu,\rho)$ are given for example
in~\cite{garsia},~\cite{kerber}. However, in the general case these algorithms become extremely involved.
We remark that the problem of computing these Kronecker multiplicities $\kappa(\lm,\mu,\rho)$ -- or
obtaining significant quantitative information about them --
is one of the last major
open problems in the ordinary representation theory of the symmetric groups.

\medskip

In this paper we consider the case where the partitions $\lm$ and $\mu$ are in the $k$-strip
$H(k,0)$, and more
generally -- in the $(k,\ell)$ hook $H(k,\ell)$. The partitions in the $k$-strip are denoted
\[
H(k,0;n)=\{\lm=(\lm_1,\lm_2,\ldots)\vdash n\mid \lm_{k+1}=0\},\,\quad\mbox{and}\quad H(k,0)=\bigcup_nH(k,0;n).
\]
Similarly, the partitions in the $(k,\ell)$-hook are denoted
\[
H(k,\ell;n)=\{\lm=(\lm_1,\lm_2,\ldots)\vdash n\mid \lm_{k+1}\le\ell\},\,
\quad\mbox{and}\quad H(k,\ell)=\bigcup_nH(k,\ell;n).
\]
We later apply the fact that as a function of $n$, the cardinality
$|H(k,\ell;n)|$ is polynomially bounded, see for
example~\cite[Theorem 7.3]{berele}.

\medskip
We mention here that these two distinct subsets of partitions, $H(k,0)$ and $H(k,\ell)$, play an 
important role in representation theory: 
By Schur's Double Centralizer Theorem, 
the partitions in $H(k,0)$ parametrize the irreducible polynomial representations of the 
General Linear Lie Group  $GL(k,\mathbb{C})$. And a similar role is played by the partitions in  
$H(k,\ell)$ and the irreducible representations of the General Linear
 Lie superalgebra $pl(k,\ell)$~\cite{berele}.

\medskip

The main results in this paper are  Theorem~\ref{main.theorem}, proved in  Section 4, and
Theorem~\ref{main.theorem.2} which is proved in Section~\ref{hook}. Theorem~\ref{main.theorem}
is a special case of Theorem~\ref{main.theorem.2}.
\begin{thm}\label{main.theorem}
Given $0<k\in\mathbb{Z}$, there exist $a=a(k)$, $b=b(k)$, satisfying the following condition:
For any $n$ and any partitions $\lm,\mu\in H(k,0;n)$  and $\rho\vdash n$, the multiplicities
$\kappa(\lm,\mu,\rho)$
of~\eqref{october.9} satisfy $\kappa(\lm,\mu,\rho)\le a\cdot n^b$. Namely, in the $k$-strip
these multiplicities $\kappa(\lm,\mu,\rho)$ are polynomially bounded.
\end{thm}
In Section~\ref{lower.3} we prove a lower bound for some multiplicities $\kappa(\lm,\lm,\nu)$,
a lower bound which grows as a polynomial of a rather large degree.

\medskip
In Section~\ref{hook} we hook-generalize Theorem~\ref{main.theorem} to the following theorem.

\begin{thm}\label{main.theorem.2}
Given $0\le k,\ell\in\mathbb{Z}$, there exist $a=a(k,\ell)$, $b=b(k,\ell)$, satisfying the following condition:
For any $n$ and any partitions $\lm,\mu\in H(k,\ell;n)$  and $\rho\vdash n$, the multiplicities
$\kappa(\lm,\mu,\rho)$
of~\eqref{october.9} satisfy $\kappa(\lm,\mu,\rho)\le a\cdot n^b$. Namely, in the $(k,\ell)$-hook
these multiplicities $\kappa(\lm,\mu,\rho)$ are polynomially bounded.
\end{thm}

One of the main tools for proving Theorem~\ref{main.theorem} is a
recursive formula for computing the multiplicities $\kappa(\lm,\mu,\rho)$, a formula
due to  Dvir~\cite[Theorem 2.3]{dvir}, and which yields a convenient upper bound for the
Kronecker multiplicities. To prove Theorem~\ref{main.theorem.2} we also need  -- and we prove -- a conjugate
version of that theorem of Dvir.

\medskip
The {\it outer} tensor product $\chi^\lm\hat\otimes\chi^\mu$, together with the Littlewood-Richardson
multiplicities $r(\lm,\mu,\nu)$, are introduced in Remark~\ref{preliminaries.1}.2. Another key tool in
proving Theorems~\ref{main.theorem} and~\ref{main.theorem.2} is the fact that in the strip and in the hook, the
multiplicities $r(\lm,\mu,\nu)$ are polynomially bounded. These properties are proved in Sections~\ref{outer}
and~\ref{hook}. We remark that
Dvir's formula~\cite[Theorem 2.3]{dvir} connects the Littlewood-Richardson and the 
Kronecker multiplicities, see~\eqref{dvir.10}.

\medskip
In Section~\ref{outside} we show that outside the hook the above Theorems~\ref{main.theorem} 
and~\ref{main.theorem.2}
fail.
In fact we show that outside the hook
some multiplicities $\kappa(\lm,\mu,\rho)$ can grow at least as fast as $\sqrt{n!}$.

\medskip

Finally we remark that it is of some interest to find out if similar phenomena -- of the multiplicities being polynomially bounded --
hold when the characteristic of the base field is finite.

\section{Preliminaries}\label{preliminaries}

The form  $<\lm,\mu>=<\chi^\lm,\chi^\mu>$ equals 1 if $\lm=\mu$, equals $0$ otherwise, 
and is extended to all characters of the symmetric groups by
bi-linearity~\cite[pg 114]{macdonald}.

\medskip
We list some facts that will be needed later.
\begin{remark}\label{preliminaries.1}
\begin{enumerate}

\item
Let $\lm=(\lm_1,\lm_2,\ldots)$ be a partition. Then $\ell(\lm)=k$ if $\lm_k>0$ and $\lm_{k+1}=0$.
For example, $\ell(\lm)\le k$ if and only if $\lm\in H(k,0)$.
\item
\begin{enumerate}

\item
Let $\f$ be a character of $S_m$, $\psi$ a character of $S_n$, with possibly $m\ne n$.
The {\it outer tensor product}
$\f\hat \otimes \psi$ is defined as follows:
$\f\times\psi$ is a character of $S_m\times S_n$, which is a subgroup of $S_{m+n}$. Inducing up,
we have
\[
\f\hat \otimes \psi=(\f\times\psi)\uparrow_{S_m\times S_n}^{S_{m+n}}.
\]
Let now $\f=\chi^\lm$ and $\psi=\chi^\mu$. Then 
$\chi^\lm\hat \otimes \chi^\mu$  is a character of $S_{m+n}$, and since $char(F)=0$, by complete reducibility
\begin{eqnarray}\label{october.30}
\chi^\lm\hat \otimes \chi^\mu=\sum_{nu\vdash n+m}r(\lm,\mu,\nu)\cdot\chi^\nu.
\end{eqnarray}
This equation defines the multiplicities $r(\lm,\mu,\nu)$.

\item

The evaluation of the multiplicities $r(\lm,\mu,\nu)$
is given by  the celebrated
Littlewood-Richardson rule, hence we call $r(\lm,\mu,\nu)$ {\it the Littlewood-Richardson multiplicities}.
In the special case that $\mu=(m)$, the decomposition of $\chi^\lm\hat\otimes \chi^{(m)}$ is
given by (the "horizontal") Young rule~\cite{kerber, macdonald, sagan}. The decomposition
of $\chi^\lm\hat\otimes \chi^{(1^m)}$ is given by the analogue "vertical" Young rule.

\item

{\it The "horizontal" Young rule}. Let $\lm=(\lm_1,\lm_2,\ldots)$ be a partition 
and $m\ge 0$ an integer. Let
$Par(\lm,m)$ denote the following set of partitions of $|\lm|+m$:
\[
Par(\lm,m)=\{\mu=(\mu_1,\mu_2,\ldots)\vdash |\lm|+m\mid \mu_1\ge\lm_1\ge \mu_2\ge\lm_2\ge\cdots\}.
\] 
then
\[
\chi^\lm\hat\otimes \chi^{(m)}=\sum_{\mu\in Par(\lm,m)}\chi^\mu.
\]

\end{enumerate}

\item
Let $\al,~\lm$ be partitions, $\al\subseteq\lm$ and consider the skew shape $\lm/\al$.
The corresponding $S_{|\lm|-|\al|}$ character $\chi^{\lm/\al}$
is defined as follows~\cite{macdonald}:
Let $\nu$ be a partition of $|\lm|-|\al|$, then
 \begin{eqnarray}\label{october.14}
 < \chi^{\lm/\al}\,,\,\chi^{\nu}>=< \chi^\lm\,,\,\chi^\al\hat\otimes\chi^\nu>.
 \end{eqnarray}
\[\mbox{Write}~~
\chi^\al\hat \otimes \chi^{\nu}=\sum_{\lm} r(\lm,\al,\nu)\cdot\chi^\lm,
~\mbox{then~\eqref{october.14} implies that}
~\chi^{\lm/\al}=\sum_\nu r(\lm,\al,\nu)\cdot\chi^\nu.
\]
Thus  the Littlewood-Richardson multiplicities $r(\lm,\al,\nu)$ 
also yield the decomposition of $\chi^{\lm/\al}$.
\item
 If $\chi^{\rho}$ is a component of $\chi^{\al}\hat \otimes \chi^{\nu}$ then
$\ell(\al),\ell(\nu)\le\ell(\rho)\le  \ell(\al)+\ell(\nu) $.
If $\al\subseteq\lm$ and $\chi^\nu$ is a component of  $\chi^{\lm/\al}$ then $r(\lm,\al,\nu)\ne 0$, so
$\chi^\lm$ is a component of
$\chi^{\al}\hat \otimes \chi^{\nu}$, hence $\ell(\nu)\le\ell(\lm)$.
\item
Let $\lm=(\lm_1,\lm_2,\ldots)$ and $\mu=(\mu_1,\mu_2,\ldots)$ be two partitions.
Then $\lm\cap\mu$ is the partition obtained by intersecting the two
corresponding Young diagrams. Thus
\[
(\lm\cap\mu)_i=min\{\lm_i,\mu_i\},~~i=1,2,\ldots.
\]
\end{enumerate}
\end{remark}

The following is a consequence of Young's rule.
\begin{lem}\label{november.3}
Let $\f$ be an $S_m$ character supported on $H(k-1,0)$:
\[
\f=\sum_{\mu\in H(k-1,0;m)}c_\mu\cdot\chi^\mu,
\]
and assume the multiplicities $c_\mu$ satisfy $c_\mu\le M$. Let $0<u$ and write
\[
\f\hat\otimes\chi^{(u)}=\sum_{\nu\in H(k,0;m+u)}d_\nu\cdot\chi^\nu.
\]
Then the multiplicities $d_\nu$ satisfy $d_\nu\le M\cdot (u+1)^k$.
\end{lem}
\begin{proof}
We have
\[
\f\hat\otimes \chi^{(u)}=\sum_{\mu\in H(k-1,0;m)}c_\mu\cdot(\chi^\mu\hat\otimes \chi^{(u)})=
\sum_{\nu\in H(k,0;m+u)}d_\nu\cdot\chi^\nu.
\]
Let $\nu\in H(k,0;m+u)$ and denote by $L$ the number of partitions $\mu\in H(k-1,0;m)$ such that 
$\chi^\nu\in \chi^\mu\hat\otimes\chi^{(u)}$. Then the multiplicity $d_\nu$ equals the sum of $L$ multiplicities 
$c_\mu$. By Young's rule 
\[
L=(\nu_1-\nu_2+1)(\nu_2-\nu_3+1)\cdots (\nu_k-\nu_{k+1}+1)
\quad\mbox{and}\quad u\ge \sum_i(\nu_i-\nu_{i+1}).
\]
(where $\nu_{k+1}=0$). Now each 
$\nu_i-\nu_{i+1}\le u$, so $L\le (u+1)^k$ and $d_\nu\le M\cdot L\le M\cdot (u+1)^k$.
\end{proof}

We shall need the following properties.
\begin{remark}\label{october.31}
Let $\lm\in H(k,\ell;n).$ Then the number of sub-partitions $\al\subseteq\lm$ is $\le (n+1)^{k+\ell}$.
In particular, if  $\lm\in H(k,0;n)$ then the number of sub-partitions $\al\subseteq\lm$ is
$\le (n+1)^{k}$.

\end{remark}
\begin{proof}
For each $1\le i\le k$ there are $\le n+1$  possible values for $\al_i$, namely the values $0,1,\ldots,n$.
Similarly for the possible values of $\al_j'$, $1\le j\le \ell$, where
$\al'$ is the conjugate partition of $\al$.
\end{proof}

\begin{prop}\label{sheli}~\cite[Theorem 3.26.a]{berele}
Let $\lm\in H(k_1,\ell_1;n)$, $\mu\in H(k_2,\ell_2;n)$, and let $k=k_1\ell_1+k_2\ell_2$
and $\ell=k_1\ell_2+k_2\ell_1$. Then $\chi^\lm\otimes\chi^\mu$ is supported
on $H(k,\ell;n)$:
\begin{eqnarray}\label{general.1}
\chi^\lm\otimes\chi^\mu=\sum_{\nu\in H(k,\ell;n)} \kappa(\lm,\mu,\nu)\cdot\chi^\nu.
\end{eqnarray}

In particular
\begin{eqnarray}\label{general.2}
\chi^\lm\otimes \chi^\mu=\sum_{\nu\in H(h,0;n)} \kappa(\lm,\mu,\nu)\cdot\chi^\nu\quad\mbox{where}
\quad h=\ell(\lm)\cdot\ell(\mu).
\end{eqnarray}

\end{prop}

For an interesting refinement of~\eqref{general.2} -- see~\cite{dvir}.

\section{Polynomial bounds in the strip for the Littlewood-Richardson multiplicities}\label{outer}
In this section we prove polynomial bounds for the Littlewood-Richardson multiplicities in the strip.
In section~\ref{hook} we indicate how to extend this to the hook.

\subsection{Polynomial upper bounds in the strip}

\begin{lem}\label{october.13}
Given non-negative integers $k_1,k_2$, there exist $a=a(k_1,k_2),~b=b(k_1,k_2)$ satisfying the
following condition:

\medskip
Let $\lm\in H(k_1,0),$  $\mu\in H(k_2,0),$
and as in Equation~\eqref{october.30},
  the Littlewood-Richardson multiplicities $ r (\lm,\mu,\nu)$ are defined by
\[
\chi^\lm\hat\otimes \chi^\mu=\sum_{\nu\vdash |\lm|+|\mu|}  r (\lm,\mu,\nu)\cdot\chi^\nu\qquad\mbox{(so $\nu\in H(k_1+k_2,0)$:
$ r (\lm,\mu,\nu)=0$ if $\ell(\nu)>k_1+k_2$)}.
\]
Then these multiplicities satisfy $ r (\lm,\mu,\nu)\le a\cdot (|\lm|+|\mu|)^b$. Namely,
in the strip, the multiplicities of $\chi^\lm\hat\otimes \chi^\mu$ are polynomially bounded.
\end{lem}

\begin{proof}
 The proof is by double induction: the first case is $(k_1,k_2)$ where $k_1$ is arbitrary and $k_2=0$.
 Then in the general case $k_2>0$, we assume true for the pair ($k_1,k_2-1$), and prove for
 ($k_1,k_2$).

 \medskip
 Note first that in the  case $k_1$ arbitrary and $k_2=0$, these multiplicities $ r (\lm,\mu,\nu)$ 
 are 1 and 0, so
 there is nothing to
 prove. Indeed, $k_2=0$ implies that $\mu$ is the empty partition $\mu=\emptyset$, then
 $\chi^\lm\hat\otimes \chi^\mu=\chi^\lm,$
 hence $r(\lm,\emptyset,\lm)=1$ and $r(\lm,\emptyset,\nu)=0$ if $\nu\ne\lm$.

\medskip
We proceed with the general case.

\medskip
By induction on the pair $(k_1, k_2-1)$, there exist $a_2,b_2>0$ satisfying the following condition:
Let $\ell(\lm)\le k_1$ and $\ell(\rho)\le k_2-1$ and write
$\chi^\lm\hat\otimes\chi^\rho=\sum_\theta r(\lm,\rho,\theta)\cdot\chi^\theta$,
then the multiplicities $r(\lm,\rho,\theta)$ satisfy  $r(\lm,\rho,\theta)\le a_2(|\lm|+|\rho|)^{b_2}$.

\medskip
It is given that $\lm\in H(k_1,0)$ and $\mu\in H(k_2,0)$.
Denote $k_2=k$, so $\mu=(\mu_1,\ldots,\mu_k)$, and denote $\bar\mu=(\mu_1,\ldots,\mu_{k-1})$. 
We have  $\chi^\lm\hat\otimes\chi^{\bar\mu}=\sum_\theta  r (\lm,\bar\mu,\theta)\cdot\chi^\theta$, then
by induction all $ r (\lm,\bar\mu,\theta)\le a_2\cdot (|\lm|+|\bar\mu|)^{b_2}$.
As in~\eqref{october.30} we write
\[
\chi^\lm\hat\otimes\chi^\mu=\sum_\nu r(\lm,\mu,\nu)\cdot\chi^\nu\quad\mbox{and we also denote}\quad
\chi^\lm\hat\otimes\chi^{\bar\mu}\hat\otimes \chi^{(\mu_k)}=\sum_\nu w(\lm,\mu,\nu)\cdot\chi^\nu.
\]
By Young's rule
$\chi^\mu$ is a component of $\chi^{\bar\mu}\hat\otimes\chi^{(\mu_k)}$ 
and therefore
$ \,r(\lm,\mu,\nu)\le w(\lm,\mu,\nu)$.
Apply now Lemma~\ref{november.3} with $\f=\chi^\lm\hat\otimes \chi^{\bar\mu}$, $u=\mu_k$, 
and $M=a_2\cdot (|\lm|+|\bar\mu|)^{b_2}$.
Each component $\chi^\nu$ of $\chi^\lm\hat\otimes\chi^{\bar\mu}$ is of length
$\ell(\nu)\le \ell(\lm)+\ell(\bar\mu)\le k_1+k_2-1$.
By Lemma~\ref{november.3}  the multiplicities $w(\lm,\mu,\nu)$ satisfy 
$w(\lm,\mu,\nu)\le a_2\cdot (|\lm|+|\bar\mu|)^{b_2}\cdot(\mu_k+1)^{k_1+k_2}$.

\medskip
For any non-negative integers $c,d,r$ and $s$,
~$r^c\cdot (s+1)^d\le (r+s)^{c+d}$, hence
\[
 r (\lm,\mu,\nu)\le w(\lm,\mu,\nu)\le a_2\cdot(|\lm|+|\bar\mu|)^{b_2}\cdot 
 (\mu_k+1)^{k_1+k_2}\le~~~~~~~~~~~~~~~~~~~~~~
 \]
 \[~~~~~~~~~~~~~~~~~~~~~~~~~~~~~~~~~
 \le a_2\cdot(|\lm|+|\bar\mu|+\mu_k)^{b_2+k_1+k_2}=a_2\cdot(|\lm|+|\mu|)^{b_2+k_1+k_2}.
\]
With $a=a_2$ and $b=b_2+k_1+k_2$ the proof of the lemma is now complete.
\end{proof}

\begin{cor}\label{pb.1}
Given $k$, there exist $a=a(k)$, $b=b(k)$ satisfying the following condition:
Let $\lm\in H(k,0)$, let $\al\subseteq \lm$ and write $\chi^{\lm/\al}=
\sum_\rho r(\lm,\al,\rho)\cdot\chi^\rho$,
then all $r(\lm,\al,\rho)\le a\cdot |\lm|^b.$  Namely the multiplicities $r(\lm,\al,\rho)$
are polynomially bounded.
 Moreover $r(\lm,\al,\rho)=0$ if $\ell(\rho)>\ell(\lm)$ or similarly  if $\rho_1>\lm_1$.
In particular $\chi^{\lm/\al}=\sum_{\rho \in H(k,0)}r(\lm,\al,\rho)\cdot\chi^\rho$.
\end{cor}
\begin{proof}
Let $k=k_1=k_2$, let $a(k,k)$ and $b(k,k)$ as in Lemma~\ref{october.13} and let $a=a(k,k)$ and
$b=b(k,k)$.

\medskip
We show first that if $\ell(\rho)>\ell(\lm)$ then
$r(\lm,\al,\rho)=0$. Indeed, let $\chi^{\lm/\al}=\sum_\rho r(\lm,\al,\rho)\cdot\chi^\rho$, then
 \[
 r(\lm,\al,\rho)=<\chi^{\lm/\al}\,,\,\chi^\rho>=<\chi^\lm\,,\,\chi^{\al}\hat\otimes\chi^{ \rho}>.
 \]
 If $\chi^\theta$ is a component of $\chi^{\al\bar\otimes \rho}$ then by Remark~\ref{preliminaries.1}.4
 $\ell(\theta)\ge \ell(\rho)>\ell(\lm),$ so $\theta\ne\lm$; hence $<\chi^\lm\,,\,\chi^\theta>=0$,
 so $0=<\chi^\lm,\chi^{\al}\hat\otimes\chi^{\rho}>=r(\lm,\al,\rho)$, namely
 $r(\lm,\al,\rho)=0$. Thus we can assume that $\rho\in H(k,0)$.
Since $\al\subseteq \lm$,  also $\al\in H(k,0)$. The multiplicities
$r(\lm,\al,\rho)$ are also the multiplicities of the irreducibles $\chi^\lm$ in $\chi^{\al\hat\otimes\rho}$, and since
$\al,\rho\in H(k,0)$, by Lemma~\ref{october.13} \[
r(\lm,\al,\rho)\le a\cdot(|\al|+|\rho|)^b=a\cdot(|\lm|)^b.
\]
\end{proof}

\subsection{Polynomial lower bounds in the strip}
The following is an example of a polynomial lower bound:

\medskip
Let $\lm=\mu=(2m,m)$, so $n=3m$, and let $\nu=(3m,2m,m)$. By direct calculations 
with the Littlewood--Richardson rule one deduces that $r(\lm,\lm,\nu)\ge m+1$, which is a polynomial
lower bound.

\medskip
This indicates that Lemma~\ref{october.13},  essentially, cannot be improved.


\section{Polynomial bounds for the Kronecker multiplicities in the strip}\label{kronecker}


We now prove a polynomial
upper bound for the Kronecker multiplicities
 $\kappa(\lm,\mu,\rho)$ (see~\eqref{october.9}), where $\lm,\mu\in H(k,0;n)$ and $\rho\vdash n$.
In Section~\ref{hook} we extend this to the hook $H(k,\ell;n)$.
We also prove here a polynomial lower bound for some $\kappa(\lm,\mu,\rho)$.

\subsection{A polynomial upper bound in the strip}

Theorem 2.3 in~\cite{dvir} gives a recursive formula for calculating $\kappa(\lm,\mu,\rho)$.
Discarding the negative term in that formula in~\cite{dvir}, it implies the following upper
bound for $\kappa(\lm,\mu,\rho)$.

\begin{thm}\label{dvir.1}
Let $\lm,\mu,\rho\vdash n$, and as in~\eqref{october.9} let
$
\chi^\lm\otimes \chi^\mu=\sum_{\rho\vdash n} \kappa(\lm,\mu,\rho)\cdot\chi^\rho,
$
then
\begin{eqnarray}\label{dvir.10}
\kappa(\lm,\mu,\rho)\le \sum_{\al\vdash \rho_1,~\al\subseteq \lm\cap\mu}
< \chi^{\lm/\al}\otimes \chi^{\mu/\al}\;, \;\chi^{(\rho_2,\rho_3,\ldots)}  >.
\end{eqnarray}
\end{thm}

{\bf The proof of Theorem~\ref{main.theorem}:}

\begin{proof}
By assumption $\lm,\mu\in H(k,0;n)$ and $\rho\vdash n$. By Proposition~\ref{sheli}
$\kappa(\lm,\mu,\rho)=0$ if $\ell(\rho)>k^2$ , hence in~\eqref{dvir.10} we can assume that $\ell(\rho)\le k^2$,
namely $\rho=(\rho_1,\ldots,\rho_{k^2})$.
By Remark~\ref{october.31}, since $\lm\in H(k,0;n)$,
in~\eqref{dvir.10} the number of sub-partitions $\al$, $\al\subseteq\lm\cap \mu\subseteq\lm$ is
$\le (n+1)^k$, which is
polynomial.  Hence
suffices to show that each summand
\[
< \chi^{\lm/\al}\otimes \chi^{\mu/\al}\;, \;\chi^{(\rho_2,\rho_3,\ldots,\rho_{k^2})}  >
\]
is polynomially bounded.

\medskip
By Corollary~\ref{pb.1} for each skew shape $\lm/\al$,  $\chi^{\lm/\al}$  is a sum
$\chi^{\lm/\al}=\sum_\pi r(\lm,\al,\pi)\cdot\chi^\pi$ where the $r(\lm,\al,\pi)$ are polynomially bounded
(polynomial in $n=|\lm|$). Moreover by Remark~\ref{preliminaries.1}.4, in that sum
$\chi^{\lm/\al}=\sum_\pi r(\lm,\al,\pi)\cdot\chi^\pi$, ~$\pi\in H(k,0;n-|\al|)=H(k,0;n-\rho_1)$.
So in particular, since $|H(k,0;n)|$ is polynomially bounded~\cite[Theorem 7.3]{berele}, that sum has
at most polynomially many summands (polynomial in $n=|\lm|$),
and the multiplicities $r(\lm,\al,\pi)$ in that sum are polynomially bounded.
Similarly for the skew shape  $\mu/\al$:  $\chi^{\mu/\al}$
is a sum of $\le$ polynomially many irreducible characters $\chi^\theta$;
again, polynomial in $n=|\lm|$, with polynomially bounded multiplicities.

\medskip
Therefore it suffices to show that for partitions
$\pi,\,\theta\in H(k,0;n-\rho_1)$, each Kronecker coefficient
$\kappa(\pi,\theta, (\rho_2,\ldots,\rho_{k^2}))$ is polynomially bounded.
Repeating one more step, deduce that it suffices to show that for any two partitions
$\tau,\,\omega\in H(k,0;n-(\rho_1+\rho_2))$, the multiplicity~$\kappa(\tau,\omega,(\rho_3,\ldots,\rho_{k^2}))$
is polynomially bounded.

\medskip
Continue! After at most $k^2$ steps we arrive at at-most polynomially many summands
$\kappa(\emptyset,\emptyset,\emptyset)=1$, and the proof is complete.
 \end{proof}
 \subsection{A polynomial lower bound in the strip}\label{lower.3}
 \begin{lem}
 Let $n=k\cdot w$, $\lm=(w,w,\ldots,w)=(w^k)\in H(k,0;n)$, fix $k$ and let $w$ go to infinity (hence also $n$
 goes to infinity). Let $\varepsilon>0$.
 Then there exist partitions $\nu\vdash n$ such that
 \[
 \kappa(\lm,\lm,\nu)\ge n^{(k^2-4)(k^2-1)/4-\varepsilon}.
 \]
 \end{lem}
 \begin{proof}
 Since $k$ is fixed and $w$ goes to infinity, by Stirling's formula, for some constant $A$
 \begin{eqnarray}\label{lower.1}
 f^\lm\simeq A\cdot\left( \frac{1}{n} \right)^{(k^2-1)/2}\cdot k^n,\quad\mbox{hence}\quad
 (f^\lm)^2\simeq A^2\cdot\left( \frac{1}{n} \right)^{k^2-1}\cdot k^{2n}.
 \end{eqnarray}
 By~\eqref{general.2}  $\chi^\lm\otimes\chi^\lm$ is supported on $H(k^2,0)$, therefore
 \[
 \chi^\lm\otimes\chi^\lm=\sum_{\nu\in H(k^2,0;n)}\kappa(\lm,\lm,\nu)\cdot
 \chi^\nu,\quad\mbox{so taking degrees we have}\quad
 (f^\lm)^2=\sum_{\nu\in H(k^2,0;n)}\kappa(\lm,\lm,\nu)\cdot f^\nu.
 \]
 Denote $g=(k^2-4)(k^2-1)/4-\varepsilon$ and assume all $\kappa(\lm,\lm,\nu) < n^g$. Then
 \[
 (f^\lm)^2<n^g\cdot \sum_{\nu\in H(k^2,0;n)}f^\lm.
 \]
 By~\cite{regev2}
 \[
 \sum_{\nu\in H(k^2,0;n)}f^\lm\simeq B\cdot\left(\frac{1}{n}  \right)^{k^2(k^2-1)/4}\cdot k^{2n},
 \]
 for some constant $B$, so
 \begin{eqnarray}\label{lower.2}
 (f^\lm)^2<n^g\cdot B\cdot\left(\frac{1}{n}  \right)^{k^2(k^2-1)/4}\cdot k^{2n}.
 \end{eqnarray}
 Combining~\eqref{lower.1} and~\eqref{lower.2}, deduce that for sufficiently large $n$,
 \[
 A^2\cdot\left( \frac{1}{n} \right)^{k^2-1}\cdot k^{2n}
 <n^g\cdot B\cdot\left(\frac{1}{n}  \right)^{k^2(k^2-1)/4}\cdot k^{2n}.
 \]
 Forming $l.h.s.\,/\,r.h.s$, deduce that for the constant $C=A^2/B$
 \[
 C\cdot n^{(k^2-4)(k^2-1)/4-g}=C\cdot n^\varepsilon<1
 \]
 for all large $n$, which is a contradiction.

 \end{proof}

\section{Polynomial upper bounds for the Kronecker multiplicities in the $(k,\ell)$-hook}\label{hook}

In this section we prove Theorem~\ref{main.theorem.2}. The proof is a hook
generalization of the proof of Theorem~\ref{main.theorem}

\medskip

Both Lemma~\ref{october.13} and Corollary~\ref{pb.1} hold in the vertical strip $H(0,\ell)$, by essentially
the same -- but conjugate -- arguments. To prove~\ref{october.13} we decomposed $\mu$ into $(\mu_1)$ and
 $\bar\mu=(\mu_2,\mu_3,\ldots)$ (namely -- first row, then the rest of $\mu$),
then applied Young's rule and the fact that $\chi^\mu$ is a component of 
$\chi^{\bar\mu}\hat\otimes\chi^{\mu_1}$.
To prove the "vertical" version of Lemma~\ref{october.13}, decompose $\mu $ into its first {\it column}
and the rest: $(1^{{\mu_1}'})$
and ${\tilde \mu}=(\mu_1-1,\mu_2-1,\ldots )$, then by the vertical Young rule
$\chi^\mu$ is a component of $\chi^{\tilde\mu}\hat\otimes \chi^{(1^{{\mu_1}'})}$. 
The rest of the arguments are the same,
yielding the vertical versions of
Lemma~\ref{october.13} and  of Corollary~\ref{pb.1}. For example, Corollary~\ref{pb.2} below is the
vertical version of Corollary~\ref{pb.1}.
\begin{cor}\label{pb.2}
Given $\ell$, there exist $a=a(\ell)$, $b=b(\ell)$ satisfying the following condition:
Let $\lm\in H(0,\ell)$, let $\al\subseteq \lm$ and write 
$\chi^{\lm/\al}=\sum_\rho r(\lm,\al,\rho)\cdot\chi^\rho$,
then all $r(\lm,\al,\rho)\le a\cdot |\lm|^b,$  namely the multiplicities $r(\lm,\al,\rho)$ 
are polynomially bounded.
 Moreover $r(\lm,\al,\rho)=0$ if $\ell(\rho)>\ell(\lm)$, or similarly  if $\rho_1>\lm_1$.
In particular $\chi^{\lm/\al}=\sum_{\rho \in H(0,\ell)}r(\lm,\al,\rho)\cdot\chi^\rho$.
\end{cor}
Combining the horizontal and the vertical versions, we deduce the $(k,\ell)$ hook versions. We now state
the $(k,\ell)$ hook version of Corollary~\ref{pb.1}:
\begin{cor}\label{pb.3}
Given $k,\ell$, there exist $a=a(k,\ell)$, $b=b(k,\ell)$ satisfying the following condition:
Let $\lm\in H(k,\ell)$, let $\al\subseteq \lm$ and write 
$\chi^{\lm/\al}=\sum_\rho r(\lm,\al,\rho)\cdot\chi^\rho$,
then all  $r(\lm,\al,\rho)$ satisfy  $r(\lm,\al,\rho)\le a\cdot |\lm|^b,$  namely the multiplicities $r(\lm,\al,\rho)$
are polynomially bounded.
 Moreover $\,r(\lm,\al,\rho)=0$ if $\rho_i>\lm_i$ or $\rho'_i>\lm'_i$ for some $i$.
In particular, since $\lm\in H(k,\ell)$ this implies that  
$\chi^{\lm/\al}=\sum_{\rho \in H(k,\ell)}r(\lm,\al,\rho)\cdot\chi^\rho$.
\end{cor}

We shall need the vertical version of the inequality~\eqref{dvir.10}, which is inequality~\eqref{dvir.14}
below.
Note that the proof of~\cite[Theorem 2.3]{dvir} was based on the decomposition of $\rho$
into its first row $(\rho_1)$
and the rest of the rows  $(\rho_2,\rho_3,\ldots)$. The second main ingredient in the proof
of~\eqref{dvir.10} was
 the horizontal Young rule, which implied that
$\chi^\rho$ is a component of $\chi^{(\rho_1)}\hat\otimes \chi^{(\rho_2,\rho_3,\ldots)}$. 
Conjugate: Decompose $\rho$ as first
column $(1^{\rho'_1})$ and the rest of the columns $(\rho_1-1\rho_2-1,\ldots ).$  By the vertical
Young rule $\chi^\rho$ is a component of  $\chi^{(1^{\rho'_1})}\hat\otimes\chi^{(\rho_1-1\rho_2-1,\ldots )}.$ 
The same arguments that
proved Theorem~\ref{dvir.1} now prove the following theorem.
\begin{thm}\label{dvir.11}
Let $\lm,\mu,\rho\vdash n$, and as in~\eqref{october.9} let
$
\chi^{\lm}\otimes \chi^{\mu}=\sum_{\rho\vdash n} \kappa(\lm,\mu,\rho)\cdot\chi^\rho.
$
Then
\begin{eqnarray}\label{dvir.14}
\kappa(\lm,\mu,\rho)\le \sum_{\al\vdash \rho'_1,~\al\subseteq \lm\cap\mu}
< \chi^{\lm/\al}\otimes \chi^{\mu/\al}\;, \;\chi^{(\rho_1-1,\rho_2-1,\ldots) } >.
\end{eqnarray}
\end{thm}

We later use the obvious fact that if $\rho\in H(0,\ell)$ then $(\rho_1-1,\rho_2-1,\ldots)\in H(0,\ell-1)$.

\medskip

We also need the fact that if $\lm,\mu\in H(k,\ell)$ then all the components of $\chi^{\lm}\otimes \chi^{\mu}$ are in
the $(k^2+\ell^2,2k\ell)$ hook; namely $\kappa(\lm,\mu,\rho)=0 $ if $\rho\not\in H(k^2+\ell^2,2k\ell)$,
see~\cite[Theorem 3.26.a]{berele}.

\medskip
{\bf Analyze} the proof of Theorem~\ref{main.theorem}. That proof describes one step (which we call 
a "D-step") in which Theorem~\ref{dvir.1} is applied to replace $\kappa(\lm,\mu,(\rho_1,\rho_2,\ldots))$
by at most polynomially many summands
of the form $\kappa(\pi,\theta,(\rho_2,\rho_3,\ldots))$. Then D-steps are applied repeatedly,
removing more and more rows of $\rho$.
 Since
$\rho\in H(k^2,0)$, after at most $k^2$ D-steps the process stops and we are done.

\medskip
Similarly, by applying Theorem~\ref{dvir.11} we have the (conjugate) D'-step: this
replaces the term $\kappa(\lm,\mu,(\rho_1,\rho_2,\ldots))$
by at most polynomially many summands
of the form $\kappa(\pi,\theta,(\rho_1-1,\rho_2-1,\ldots))$.
Thus, given $\kappa(\lm,\mu,\rho)$,
a D'-step removes the first column of $\rho$. The crucial fact here is, that the condition of
polynomially bounded is preserved in ether a D or a D' step.

\medskip
{\bf The proof of Theorem~\ref{main.theorem.2}}
\begin{proof}
Start with $\kappa(\lm,\mu,\rho)$.
Since $\lm,\mu\in H(k,\ell)$, by Proposition~\ref{sheli} we can assume that
$\rho\in H(k^2+\ell^2,2k\ell)$. Apply D-steps repeatedly
until $\rho$  is replaced by $\rho^*$ where $\rho^*\in H(0,2k\ell)$.
Since $\rho\in H(k^2+\ell^2,2k\ell)$, this happens after at most $k^2+\ell^2$ D-steps. Now apply
 D'-steps repeatedly until  $\kappa(\emptyset,\emptyset,\emptyset)=1$ is reached. 
 Since $\rho^*\in H(0,2k\ell)$,
this happens after at most $2k\ell$ ~D'-steps. Thus, after a total of at most
$(k^2+\ell^2)+2k\ell=(k+\ell)^2$  steps
of types D and D', we arrive at at-most polynomially many summands, each equals 
$\kappa(\emptyset,\emptyset,\emptyset)=1$, and that completes the proof.
\end{proof}

\section{Outside the hook}\label{outside}

We now give examples outside the hook, were the
Littlewood-Richardson and the Kronecker multiplicities are {\it not}
polynomially bounded.
\subsection{ A $\sqrt{n!}$ lower bounds for some Kronecker multiplicities}
\begin{example}

Here we show that outside the hook, as $n$ goes to infinity some Kronecker multiplicities
grow at least as fast as $(n/e)^{n/2}$ namely as $\sqrt{n!}$.
In particular these  multiplicities are not polynomially bounded.
So let $\varepsilon>0$, assume that as $n$ goes to infinity all Kronecker multiplicities
$\kappa(\lm,\mu,\nu)$ satisfy
\begin{eqnarray}\label{28}
\kappa(\lm,\mu,\nu)<\left( \frac{n}{e}\right)^{\frac{n}{2}(1-\varepsilon)},
\end{eqnarray}
and we derive a contradiction.

\medskip
As usual we denote $\deg(\chi^\lm)=f^\lm$.
Let  $\lm=\mu\vdash n$ be the Vershik-Kerov Logan-Shepp partition which
maximizes $f^\lm$~\cite{logan},~\cite{vershik}.
It is known that for these partitions $\lm$ there exist constants $c_0,\,c_1>0$ such that
\[
e^{-c_1\sqrt{n}}\cdot \sqrt{n!}\le \deg(\chi^\lm)\le e^{-c_0\sqrt{n}}\cdot \sqrt{n!},
\]
and by a slight abuse of notations we write
\begin{eqnarray}\label{19}
\deg(\chi^\lm)\simeq e^{-c\sqrt{n}}\cdot \sqrt{n!}.
\end{eqnarray}
By squaring we similarly have
\begin{eqnarray}\label{20}
\deg (\chi^\lm\otimes\chi^\lm)=((\deg(\chi^\lm))^2\simeq e^{-C\sqrt{n}}\cdot n!,
\end{eqnarray}
 $C=2c>0$ a constant.

 \medskip
On the other hand assumption~\eqref{28} implies that
\begin{eqnarray}\label{21}
\deg (\chi^\lm\otimes\chi^\lm)< \left( \frac{n}{e}\right)^{\frac{n}{2}(1-\varepsilon)}
\sum_{\nu\vdash n}\deg(\chi^\nu).
\end{eqnarray}

It follows from the RSK correspondence~\cite{stanley} that the sum
$\sum_{\nu\vdash n}\deg(\chi^\nu)$ equals $T_n$, the number of involutions
in $S_n$. It was proved in~\cite{herstein} that
\[
T_n\simeq\frac{(n/e)^{n/2}\cdot e^{\sqrt{n}}}{\sqrt{2}\cdot e^{1/4}}.
\]
By Stirling's formula \[
\left(\frac{n}{e}  \right)^n\simeq\frac{1}{\sqrt{\pi n}}\cdot n!
\]
hence
\[
T_n=\sum_{\nu\vdash n}\deg(\chi^\nu)\simeq \frac{e^{\sqrt{n}}\cdot\sqrt{n!}}{({\pi n})^{1/4}\cdot q}
\qquad\mbox{where}\quad
q=\sqrt{2}\cdot e^{1/4}.
\]
so
\begin{eqnarray}\label{22}
T_n<e^{\sqrt{n}}\cdot\sqrt{n!}
\end{eqnarray}
By~\eqref{20},~\eqref{21} and~\eqref{22}
\begin{eqnarray}\label{31}
e^{-C\sqrt{n}}\cdot n!< \left( \frac{n}{e}\right)^{\frac{n}{2}(1-\varepsilon)}
e^{\sqrt{n}}\cdot\sqrt{n!}
\qquad\mbox{so}\qquad
 \sqrt{n!}< \left( \frac{n}{e}\right)^{\frac{n}{2}(1-\varepsilon)}
 \cdot e^{(C+1)\sqrt{n}}.
\end{eqnarray}
From Stirling's formula deduce that $(n/e)^{n/2}<\sqrt{n!}$, therefore~\eqref{31} implies that
\begin{eqnarray}\label{32}
\left( \frac{n}{e}\right)^{n/2}<\left( \frac{n}{e}\right)^{\frac{n}{2}(1-\varepsilon)}\cdot e^{(C+1)\sqrt{n}},
\quad\mbox{hence}\quad \left( \frac{n}{e}\right)^{\frac{n}{2}\varepsilon}<e^{(C+1)\sqrt{n}}.
\end{eqnarray}
This is a contradiction since the right hand side is sub-exponential, while the
left hand side is essentially $(\sqrt{n!})^{\varepsilon}$, which grows to
infinity faster than any exponential.
\end{example}
We conclude this section with the following conjecture.

\begin{conjecture}
For partitions of $n$ ~$\lm,\mu,\rho\vdash n$,
all Kronecker multiplicities $\kappa(\lm,\mu,\rho)$ are bounded above by $\sqrt{n!}$.
\end{conjecture}

\subsection{Exponential lower bound for some Littlewood-Richardson multiplicities}
\begin{example}
Based on that same Vershik-Kerov Logan-Shepp partitions $\lm$,
we now give an example where the Littlewood-Richardson multiplicities are not
bounded by exponential growth $a^n$ for any $a<2$.
Assume that as $n$ goes to infinity, the
 Littlewood-Richardson multiplicities in $\lm\hat\otimes\lm$ indeed are bounded by $a^n$ for some $a<2$.
 Then for large $n$
 \[
 \deg(\chi^\lm\hat\otimes\chi^\lm)< a^n\sum_{\nu\vdash 2n}\deg(\chi^\nu).
 \]
 Replacing $n$ by $2n$ in~\eqref{22} we have
 \begin{eqnarray}\label{23}
 \sum_{\nu\vdash 2n}\deg(\chi^\nu)\simeq \frac{e^{\sqrt{2n}}\cdot\sqrt{(2n)!}}{({2 \pi n})^{1/4}\cdot q}
\quad\mbox{where}\quad
q=\sqrt{2}\cdot e^{1/4}.
 \end{eqnarray}
 Thus for large $n$
 \begin{eqnarray}\label{24}
 \deg(\chi^\lm\hat\otimes\chi^\lm)< a^n\cdot e^{\sqrt{2n}}\cdot \sqrt{(2n)!}.
 \end{eqnarray}

In general (see for example~\cite{kerber}),
if $\f$ is an $S_m$ character and $\psi$  is an $S_n$ character, then
\[
\deg(\f\hat\otimes\psi)=
{m+n\choose n}\deg(\f)\deg(\psi).
\]

Therefore
\begin{eqnarray}\label{25}
\deg(\chi^\lm\hat\otimes\chi^\lm)=(\deg(\chi^\lm))^2\cdot {2n\choose n}.
\end{eqnarray}
By~\eqref{19} $\deg(\chi^\lm)\simeq e^{-c\sqrt{n}}\cdot \sqrt{n!}$
~hence
\[
\deg(\chi^\lm\hat\otimes\chi^\lm )\simeq e^{-2c\sqrt{n}}\cdot n!\cdot{2n\choose n}=
e^{-2c\sqrt{n}}\cdot \frac{(2n)!}{n!}.
\]
Thus~\eqref{24} and~\eqref{25} imply that
\[
e^{-2c\sqrt{n}}\cdot \frac{(2n)!}{n!}< a^n\cdot e^{\sqrt{2n}}\cdot \sqrt{(2n)!},
\]
so
\begin{eqnarray}\label{26}
\frac{(2n)!}{n!\cdot \sqrt{(2n)!}}< a^n\cdot e^{(2c+2)\sqrt{n}}.
\end{eqnarray}
Squaring both sides we get that
\begin{eqnarray}\label{27}
{2n\choose n }< (a^2)^n\cdot e^{4(c+1)\sqrt{n}}.
\end{eqnarray}
By Stirling's formula
\[
{2n\choose n }\simeq\frac{\sqrt 2}{\sqrt{\pi n}}\cdot 4^n,
\]
hence~\eqref{27} implies that
\[
\left(\frac{4}{a^2}\right)^n<\sqrt{\frac{\pi n}{2}}\cdot e^{4(c+1)\sqrt n}.
\]
Since $a^2<4$, the left hand side grows exponentially with $n$, while the right
grows sub-exponentially, hence a contradiction.
\end{example}

{\bf Amitai Regev}, Mathematics Department, The Weizmann Institute, Rehovot 76100 Israel.

email: amitai.regev at weizmann.ac.il


\begin{thebibliography}{99}

\bibitem{berele} A. Berele and A. Regev, Hook Young diagrams with applications to
combinatorics and to representations of Lie superalgebras, {\it Adv. Math.} {\bf 64} (1987) 118-175.

 \bibitem{herstein}  S. Chowla,  I. N. Herstein and W. K. Moore,  On recursions connected
 with symmetric groups. {\it Canadian J. Math.} {\bf 3}, (1951). 328–-334.

\bibitem{dvir} Y. Dvir, On the Kronecker product of $S_n$ characters, {\it J. Algebra}
{\bf 154} (1993), 125-140.

\bibitem{garsia} A. Garsia and J. Remmel, Schuffles of permutations and the Kronecker product,
{\it Graphs Combin.} {\bf 1} (1985), 217-263.



\bibitem{kerber} G. James and A. Kerber, The Representation Theory of the Symmetric Group,
{\it Encyclopedia of Mathematics and its Applications}, Vol {\bf 16}, Cambridge Univ. Press, 1984.

\bibitem{logan} B. F. Logan and L. A. Shepp, A variational problem for random Young tableaux,
{\it Adv. Math.},
{\bf 26} (1977) 206-222.

\bibitem{macdonald} I. G. Macdonald, Symmetric Functions and Hall Polynomials, Second edition, {\it Oxford Mathematical Monographs} 1995.



\bibitem{regev2} A. Regev, Asymptotic values for degrees associated with strips of Young diagrams,
{\it Adv. Math.} {\bf 41} (1981), 115-136.

\bibitem{sagan} B. Sagan, The Symmetric Group, Secon Edition, {\it Springer Graduate Texts in
Mathematics}, 2000.

\bibitem{stanley} R. P. Stanley, Enumerative Combinatorics Vol 2, {\it Cambridge Univ. Press}, 1999.


\bibitem{vershik}
 A. M. Vershik, and  S. V. Kerov,
Asymptotic behavior of the maximum and generic dimensions of irreducible representations
of the symmetric group. (Russian)
{\it Funktsional. Anal. i Prilozhen.} {\bf 19} (1985), no. 1, 25–-36, 96. English translation:
{\it Funct. Anal. Appl.} {\bf 19} (1985), 21--31.
\end{thebibliography}
\end{document}